\documentclass[11pt]{article}

\usepackage[margin=1in]{geometry}
\usepackage{amsmath,amssymb,amsthm,mathtools}
\usepackage{enumitem}
\usepackage{xcolor}
\usepackage{hyperref}
\usepackage{comment}
\usepackage[nameinlink,capitalize,noabbrev]{cleveref}

\hypersetup{
  colorlinks=true,
  linkcolor=blue!60!black,
  citecolor=blue!60!black,
  urlcolor=blue!60!black
}

\newtheorem{theorem}{Theorem}[section]
\newtheorem{proposition}[theorem]{Proposition}
\newtheorem{lemma}[theorem]{Lemma}
\newtheorem{corollary}[theorem]{Corollary}

\newtheorem{definition}[theorem]{Definition}
\newtheorem{conjecture}[theorem]{Conjecture}
\theoremstyle{remark}
\newtheorem{remark}[theorem]{Remark}

\newcommand{\F}{\mathcal F}

\newcommand{\W}{\mathcal W}

\newcommand{\1}{\mathbf 1}
\newcommand{\sr}{\operatorname{sr}}
\newcommand{\rank}{\operatorname{rank}}

\newcommand{\Span}{\operatorname{span}}
\newcommand{\QQ}{\mathbb Q}

\title{A Polynomial Improvement of Naslund--Sawin Bound for Sunflower-Free Families Using Triangular Tensors}
\author{Omran Ahmadi$^\dagger$ \and Hassan Norouzi$^*$}

\date{
\vspace*{3mm}
\parbox{\linewidth}{
\centering
\small
 $^\dagger$ School of Mathematics,\\ 
  Institute for Research in Fundamental Sciences (IPM),\\ 
  E-mail: \texttt{oahmadid@ipm.ir}
 \endgraf\medskip
  $^*$ School of Mathematics,\\ 
  Institute for Research in Fundamental Sciences (IPM),\\
  E-mail: \texttt{norouzi@ipm.ir}
  \endgraf\medskip
  \today
  }
}

\begin{document}
\maketitle

\begin{abstract}
Naslund and Sawin used the slice-rank method for diagonal tensors to prove that
\[
  |\F|
  =O\!\left(n^{1/2}\left(\frac{3}{2^{2/3}}\right)^n\right)
\]
 for any sunflower-free family \(\F\subseteq 2^{[n]}\). We prove a lemma similar to the slice-rank lemma for the newly-defined $i$-triangular tensors, and use it to achieve a polynomial-factor improvement of the bound of Naslund and Sawin by proving that 

\[
  |\F|=O\!\left(n^{1/6}\left(\frac{3}{2^{2/3}}\right)^n\right)
\]
for any sunflower-free family \(\F\subseteq 2^{[n]}\).
\end{abstract}


\section{Introduction}
Let $[n]:=\{1,2,\cdots,n\}$. For $k\ge 3$, a collection $\mathcal{F}$ of subsets of $[n]$ is said to contain a {\it{$k$-sunflower}} if there are $k$ pairwise distinct sets in $\mathcal{F}$ such that the intersection of any pair of them is the same. A family $\mathcal{F}$ is said to be {\it{$k$-sunflower free}} if it does not contain a $k$-sunflower. The notion of a sunflower in set-systems was introduced by Erd\H{o}s and Rado in~\cite{ErdosRado}, where the following conjecture was made.

\begin{conjecture}[Erd\H{o}s-Rado Sunflower Conjecture]
For $k\ge 3$, let $\mathcal{F}$ be a $k$-sunflower free family of $m$-subsets of $[n]$, i.e., each set is of size $m$. Then 
\[
|\mathcal{F}|\le \mu^m
\]
for a constant $\mu$ depending only on $k$.
\end{conjecture}

Later, Erd\H{o}s and Szemer\'{e}di~\cite{ErdosSzemeredi1978} made the following non-uniform version of the above conjecture. 
\begin{conjecture}[Erd\H{o}s-Szemer\'{e}di Sunflower Conjecture]
 For $k\ge 3$, let $\mathcal{F}$ be a $k$-sunflower free family of subsets of $[n]$. Then 
\[
|\mathcal{F}|\le \lambda^n
\]   
for a constant $\lambda<2$ depending only on $k$.
\end{conjecture}
In~\cite{NaslundSawin}, Naslund and Sawin used the slice-rank method for diagonal tensors to settle the above conjecture by proving that: 
\begin{theorem}\label{Naslund-Sawin-Theorem}
Let $\mathcal{F}$ be a sunflower-free family of subsets of $[n]$. Then
\[
|\F|\le 3n\sum_{k\le n/3}\binom nk
  =O\!\left(n^{1/2}\left(\frac{3}{2^{2/3}}\right)^n\right).
\]

\end{theorem}
    
The slice-rank method for diagonal tensors which is based on the slice-rank lemma of Tao for diagonal tensors~\cite{TaoWeblog} is a polynomial method rooted in a breakthrough paper by Lev, Croot and Pach~\cite{croot_progression-free_2016} and subsequent paper of Ellenberg and Gijswijt~\cite{ellenberg2016large}. 

The slice-rank lemma for diagonal tensors has been generalized to other types of tensors, see for example~\cite{Amanov,AhmadiNorouzi}. In~\cite{AhmadiNorouzi}, we introduced the notion of higher-order $i$-triangular tensors and proved an analogue of Tao's slice-rank lemma for such tensors of even order. In this paper, first we extend the corresponding result to $i$-triangular tensors of any order, then, we apply it to prove our main result given below which is a polynomial improvement on the Naslund-Sawin bound of Theorem~\ref{Naslund-Sawin-Theorem}. 

\begin{theorem}\label{MainTheorem}
 Let $\mathcal{F}$ be a sunflower-free family of subsets of $[n]$. Then
\[
|\F|\le O\!\left(n^{1/6}\left(\frac{3}{2^{2/3}}\right)^n\right).
\]   
\end{theorem}

This paper is organized as follows. In Section~\ref{LinearAlgebra}, we briefly explain the slice-rank method and then the slice-rank lemma for $i$-triangular tensors is proved. Section~\ref{Naslund-Sawin} contains a brief account of the Naslund-Sawin theorem. In Section~\ref{Triangular}, our 1-triangular tensor for sunflower-free families of sets is presented, and in Section~\ref{SliceRankEstimate} the bound of Theorem~\ref{MainTheorem} is derived. 

\section{$i$-Triangular slice-rank lemma}\label{LinearAlgebra}
Let $A_i$, $i=1,\ldots,k$, be some finite sets and $\mathbb{F}$ be a field. A $k$-tensor $T$ is a function from $\prod_i A_i$ to $\mathbb{F}$. A tensor $T$ is called a {\it{slice}} if it can be decomposed as \(T = f g\) where $f$ is a 1-tensor from $A_j$ to $\mathbb{F}$ for some $j$ in $[k]$ and $g$ is a $(k-1)$-tensor from $\prod_{i\neq j}A_i$ to $\mathbb{F}$. The important notion of the slice-rank of a $k$-tensor is as follows. 
\begin{definition}[Slice-rank]
    Let $T\colon \prod A_i \longrightarrow \mathbb{F}$ be a $k$-tensor. The slice-rank denoted $\operatorname{sr}(T)$ of $T$ is the minimum $r$ such that $T$ can be written as the sum of $r$ slices, i.e., $\operatorname{sr}(T)$ is the minimum number $r$ such that
    \[
        T= \sum_{i=1}^r f_i g_i
    \]
    where for every $i$, $f_i$ is a 1-tensor from $A_j$ to $\mathbb{F}$ for some $j$ in $[k]$ and $g_i$ is a $(k-1)$-tensor from $\prod_{i\neq j}A_i$ to $\mathbb{F}$ .
\end{definition}

When $A_1=\cdots=A_k$, a $k$-tensor $T$ is diagonal if
\(T(a_1,\ldots,a_k)\neq 0\)
implies that
\(a_1=\cdots=a_k\).

The following lemma has been at the heart of the slice-rank method with diagonal tensors. 
\begin{lemma}[T. Tao \cite{SawinTaoNotes}]\label{diagonalTensor}
    \label{slicerank}
	Let $A$ be a finite set and $\mathbb{F}$ be a field. If $T\colon A^k \rightarrow \mathbb{F}$ is a diagonal $k$-tensor  with non-zero diagonal entries, then $\operatorname{sr}(T)=|A|$.
\end{lemma}

As the notions of slice-rank and usual rank coincide for matrices, i.e. 2-tensors, the above lemma can be considered as the natural generalization of the fact that square-matrices with non-zero diagonal entries are full rank. Thus, naturally, one wonders whether it is possible to generalize the other facts about the ranks of special matrices to higher-order tensors. What follows in this section can be viewed of a generalization of the fact that triangular matrices with non-zero diagonal entries have full rank. As it was mentioned earlier, in~\cite{AhmadiNorouzi}, we introduced the notion of higher-order $i$-triangular tensors and proved a lemma similar to the slice-rank lemma of Tao for even-order tensors. Here, the corresponding lemma, Lemma~\ref{TriangularTensor} in this paper, is proved for tensors of every order. First, we recall the definition of an $i$-triangular tensor. 

\begin{definition}\label{i-Triangular}
 Let $(A,\le)$ be a totally ordered finite set, and $\mathbb{F}$ be a field. A $k$-tensor $T:A^k\to \mathbb{F}$ is $i$-triangular if $T(x_1,\ldots,x_k)=0$ whenever $(x_1,\ldots,x_k)$ is off-diagonal and $x_i\le x_j$ for all $j\ne i$.    
\end{definition}

Our tool to prove our main result in this section is the following functional reformulation of Proposition 4 appearing in~\cite{SawinTaoNotes}.

\begin{lemma}
\label{tao-cover-lemma}
Let $(A_i,\le_i)$ for $i=1,\ldots,k$ be totally ordered finite sets, and $\preceq$ be the partial order on $\prod_{i=1}^kA_i$ induced from the orders $\le_i$ on $A_i$, i.e., $(a_1,\ldots,a_k)\preceq (b_1,\ldots,b_k)$ if $a_i\le_i b_i$ for every $i$. Furthermore, let $\mathbb{F}$ be a field, $T:\prod_{i=1}^kA_i\to \mathbb{F}$ be a $k$-tensor, 
\[
\operatorname{supp}(T)=\left\{x\in \prod_{i=1}^kA_i : T(x)\neq 0\right\}, 
\]
and $\Gamma\subseteq \operatorname{supp}(T)$ be the set of maximal elements of $\operatorname{supp}(T)$. Then
\[
\operatorname{sr}(T)\ge\operatorname{\min_{\substack{\Gamma_1,\cdots,\Gamma_k\\\Gamma=\bigcup_{j=1}^k\Gamma_j}}}\sum_{j=1}^k|\pi_j(\Gamma_j)|
\]
where the minimum is taken over all $\Gamma_1,\ldots,\Gamma_k\subseteq\prod_{i=1}^kA_i $ such that 
\[
\Gamma=\bigcup_{j=1}^k\Gamma_j,
\]
and $\pi_j(\cdot)$ denotes the projection onto the $j$-th coordinate. 
\end{lemma}

\begin{lemma}\label{TriangularTensor}
Let $(A,\leq)$ be a totally ordered finite set, $k\geq 2$, and let
$i\in[k]$. Suppose 
\[
T\colon A^k\longrightarrow \mathbb{F}
\]
is $i$-triangular with non-zero diagonal entries. Then
\[
\operatorname{sr}(T)=|A|.
\]
\end{lemma}

\begin{proof}
 We equip the $i$-th copy of $A$ with
the reverse of $\leq$, and every other copy with $\leq$, and consider
the resulting product order $\preceq$ on $A^k$, i.e., $(a_1,\ldots,a_i,\ldots,a_k)\preceq (b_1,\ldots,b_i,\ldots,b_k)$ if $b_i\le a_i$ and for every $j\neq i$, $a_j\le b_j$. Now, let $\Gamma \subseteq \operatorname{supp}(T)$ be the set of maximal elements of $\operatorname{supp}(T)$.

We claim that every diagonal point
\[
\delta_a:=(a,\ldots,a),\qquad a\in A,
\]
is in $\Gamma$. If
\[
\delta_a\preceq (b_1,\ldots,b_k)\in \operatorname{supp}(T),
\]
then $b_i\leq a$ and $a\leq b_j$ for every $j\neq i$. In particular, $b_i\leq b_j$ for all $j\neq i$. Since
$(b_1,\ldots,b_k)\in\operatorname{supp}(T)$ and hence $T(b_1,\ldots,b_k)\neq 0$, $i$-triangularity implies that $(b_1,\ldots,b_k)$ is a diagonal element and thus
$b_1=\cdots=b_k=b$ for some $b$. The preceding inequalities imply
$b\leq a\leq b$, and hence $b=a$. Thus
$(b_1,\ldots,b_k)=\delta_a$, proving the claim.

Now, with the notation of Lemma~\ref{tao-cover-lemma}, if
\[
\Gamma=\Gamma_1\cup\cdots\cup\Gamma_k,
\]
then each $\delta_a$ belongs to some $\Gamma_j$, and therefore
$a\in\pi_j(\Gamma_j)$. Consequently,
\[
A\subseteq\bigcup_{j=1}^k\pi_j(\Gamma_j),
\qquad\text{so}\qquad
\sum_{j=1}^k|\pi_j(\Gamma_j)|\geq |A|.
\]
Thus, from the above lemma we have
\[
\operatorname{sr}(T)\geq |A|.
\]
The reverse inequality is immediate from the trivial slicing
\[
T(a_1,\ldots,a_k)
 =
 \sum_{a\in A}
 \mathbf{1}_{\{a_i=a\}}\,
 T(a_1,\ldots,a_{i-1},a,a_{i+1},\ldots,a_k).
\]
Hence
\[
\operatorname{sr}(T)=|A|.
\qedhere
\]
\end{proof}

\section{Naslund--Sawin 3-tensor}\label{Naslund-Sawin}

In this section, we briefly present Naslund and Sawin's proof of Theorem~\ref{Naslund-Sawin-Theorem}.
First, we notice that in order to prove that a collection \(\mathcal{F}\) of subsets of $[n]$ is sunflower-free, it suffices to prove that it is 3-sunflower free; i.e., there are not three distinct sets \(A,B,C\in\mathcal{F}\) such that 
\[
A\cap B=A\cap C=B\cap C.
\] 

Now, let $\mathcal{F}$ be sunflower-free, identify each set \(X\subseteq[n]\) with its characteristic vector
\[
x=(x_1,\dots,x_n)\in\{0,1\}^n,
\]
and consider the $3$-tensor 
\[
P(X,Y,Z)=\prod_{i=1}^n(2-x_i-y_i-z_i)
\]
from $\mathcal{F}^3$ to $\mathbb{Q}$. As $\mathcal{F}$ is sunflower-free and hence there are not three distinct sets in $\mathcal{F}$ with pairwise equal intersection and furthermore $P(X,Y,Z)=0$ if and only if exactly two of $x_i,y_i,z_i$ are equal for some $i$, the tensor $P$ satisfies the following conditions:
\begin{itemize}
    \item $P(X,Y,Z)=0$ if $X,Y$ and $Z$ are pairwise distinct,
    \item $P(X,X,Y)\neq 0$ whenever $X\subseteq Y$.
\end{itemize}

The second condition implies that the tensor $P$ is not diagonal if there are $X$ and $Y$ in $\mathcal{F}$ such that $X\subset Y$. So, one cannot apply the slice-rank Lemma~\ref{diagonalTensor}. To work around this obstacle, Naslund and Sawin consider the tensor $P$ over the subsets $\mathcal{F}_r$ of $\mathcal{F}$, $r=0,1,\ldots,n$, where 
\[
\mathcal F_r=\{A\in\mathcal F:|A|=r\}.
\]
It follows that $P$ is diagonal over $\mathcal{F}_r^3$ for every $r$ and hence using Lemma~\ref{diagonalTensor}:
\[
|\mathcal F_r|
\leq \operatorname{sr}(P).
\]

As every monomial 
\[x_Iy_Jz_K=\prod_{i\in I}x_i\prod_{j\in J}y_j\prod_{k\in K}z_k\]
in the expansion of \(P\) is of total degree at most \(n\), one of $I,J$ and $K$ is of size at most
\(n/3\), and hence slicing the monomials through such parts yields
\[
|\mathcal{F}_r|=\operatorname{sr}(P)\le 3\sum_{j=0}^{\lfloor n/3\rfloor}\binom{n}{j}
\]
for every $r$. Finally, since for some $r$, \(|\mathcal{F}_r|\ge |\mathcal{F}|/(n+1)\), it follows that
\[
|\mathcal F|
\leq
3(n+1)\sum_{j=0}^{\lfloor n/3\rfloor}\binom{n}{j}.
\]

\section{1-Triangular sunflower tensor}\label{Triangular}
In this section, we derive $1$-triangular tensor related to sunflower-free family of sets. 
We order the subsets of $[n]$ by increasing size, with an arbitrary tie-breaking order for the sets of the same size.  Thus
\[
  A\preceq B
\]
implies that either \(|A|<|B|\), or \(|A|=|B|\) and \(A\) precedes \(B\) in the fixed tie-breaking order.

Let
\[
  E(B,C)=\1_{\{|B|=|C|\}}
\]
and
\[
  T(A,B,C)=P(A,B,C)E(B,C)
\]
where the $3$-tensor $P$ is the same tensor from the previous section. We have: 
\begin{proposition}\label{prop:triangularity}
Let \(\F\subseteq2^{[n]}\) be sunflower-free.  Then \(T\colon\F^3\longrightarrow \mathbb{Q}\) is $1$-triangular with non-zero diagonal entries with respect to the order \(\preceq\).
\end{proposition}

\begin{proof}
First, we notice that 
\[
  T(A,A,A)=P(A,A,A)E(A,A)=\prod_{i=1}^n(2-3a_i).
\]
where each factor is either \(2\) or \(-1\) implying that the diagonal entries are non-zero.

Now, suppose \((A,B,C)\) is not diagonal and \(A\preceq B\), \(A\preceq C\). If \(A,B,C\) are pairwise distinct, then, as in previous section \(P(A,B,C)=0\) and hence \(T(A,B,C)=0\). It remains to prove the claim for the cases where two of $A,B$ and $C$ are the same. Let \(A=B\) and \(C\neq A\), and hence
\[
  T(A,A,C)=P(A,A,C)E(A,C).
\]
If \(|C|>|A|\), then \(E(A,C)=0\).  Otherwise \(|C|=|A|\) and  \(T(A,A,C)\neq0\) would imply that there is no coordinate $i$ for which \((a_i,a_i,c_i)=(1,1,0)\) which in turn implies that \(A\subseteq C\). But \(|A|=|C|\) implying that \(A=C\) which is a contradiction.  Hence \(T(A,A,C)=0\).  The case \(A=C\) and \(A\neq B\) can be dealt with similarly.

Finally, suppose \(A\neq B=C\).  If \(P(A,B,B)\neq0\) and hence \(T(A,B,B)\neq 0\), then there is no coordinate for which \((a_i,b_i,b_i)=(0,1,1)\), and hence \(B\subseteq A\).  But this is in a contradiction with the fact that \(A\preceq B\). 
\end{proof}

\begin{corollary}\label{cor:triangular-lower}
For every sunflower-free family \(\F\subseteq2^{[n]}\),
\[
  |\F|\le \operatorname{sr}(T).
\]
\end{corollary}

\section{Estimating the slice-rank}\label{SliceRankEstimate}

Let $\mathbb{F}$ be a field, and $\mathbb{F}[x_1,\ldots,x_n]$ denote the polynomial ring over $\mathbb{F}$ in the variables $x_1,\ldots,x_n$. For every set $S\subseteq [n]$, let $x_S=\prod_{i\in S}x_i$ be the square-free monomial corresponding to $S$. For $r\le n$, $N(r)=\sum_{i=0}^r\binom ni$ is the number of square-free monomials of degree at most $r$. If we identify each set \(A\subseteq[n]\) with its characteristic vector
\[
a=(a_1,\dots,a_n)\in\{0,1\}^n,
\]
in which $a_i=1$ if and only if $i\in A$, then we may consider $x_S$ as a function from the subsets of $[n]$ to $\{0,1\}$ if we let $x_S(A)=\prod_{i\in S}a_i$, i.e., \(x_S(A)=\1_{\{S\subseteq A\}}\) for $A\subseteq [n]$. 

In order to give an upper-bound on the slice-rank of the 1-triangular tensor of the previous section,  we expand $P(A,B,C)$ in $T(A,B,C)$ and notice that every monomial in the expansion is of the form
\[
  a_I b_J c_K
  =\prod_{i\in I}a_i\prod_{j\in J}b_j\prod_{k\in K}c_k,
\]
for some pairwise disjoint \(I,J,K\subseteq[n]\) where
\[
  |I|+|J|+|K|\le n.
\]
Now, as
\[
  E(B,C)=\sum_{\ell=0}^n\1_{\{|B|=\ell\}}\1_{\{|C|=\ell\}},
\] with the notations as before we may consider the functions $x_{I_q}(A)=\prod_{i\in I_q}a_i$, $x_{J_q}(B)=\prod_{j\in J_q}b_j$, $x_{K_q}(C)=\prod_{k\in K_q}c_k$, \(1_{\{|B|=\ell\}}\), and \(\1_{\{|C|=\ell\}}\) from the subset of \([n]\) to \(\{0,1\}\), expand \(T(A,B,C)\) 
and get the following slice decomposition of $T(A,B,C)$:
\begin{align}
 T(A,B,C)
={}&
 \sum_{q=0}^{r_1-1}
 \lambda_qx_{I_q}(A)\cdot \bigl(x_{J_q}(B)x_{K_q}(C)E(B,C)\bigr)
 \notag\\
&+
 \sum_{\ell=0}^{n}
 \sum_{q=r_1}^{r_1+r_2-1}\gamma_{l,q}
 \Bigl(\1_{\{|B|=\ell\}}x_{J_q}(B)\Bigr)
 G_{\ell,q}(A,C)
 \notag\\
&+
 \sum_{\ell=0}^{n}
  \sum_{q=r_1+r_2}^{r_1+r_2+r_3-1}\theta_{l,q}
 \Bigl(\1_{\{|C|=\ell\}}x_{K_q}(C)\Bigr)
 H_{\ell,q}(A,B)
 \label{eq:three-slice}
\end{align}
for some constants \(\lambda_q,\gamma_{l,q} \) and \(\theta_{l,q}\), and non-negative integers $r_1,r_2$ and $r_3$. 
As the dimension of the vector spaces generated by the functions \(x_{I_q}(A), \1_{\{|B|=\ell\}}x_{J_q}(B)\) and \(\1_{\{|C|=\ell\}}x_{K_q}(C)\) gives an upper bound for the slice-rank of \(T(A,B,C)\), we are going to determine them in the sequel. 

\subsection{Dimension counts}

For $S\subseteq [n]$, as before let $x_S$ denote the function that maps a subset $A$ of $[n]$ to $x_S(A)$, and, furthermore, for $0\le \ell\le n$, let $g_{l,S}$ denote the function which maps a subset $A$ of $[n]$ to $\1_{\{|A|=\ell\}}x_S(A)$. The following lemmas give the dimension of the certain subspaces of the space of functions generated by the functions $x_S$ and $g_{l,S}$. The proof of the first lemma is trivial and hence we omit it here. 

\begin{lemma}\label{Monomial-Count}
  Let \(\mathbb{F}\) be a field of characteristic zero, and
\[
U_r=\operatorname{span}\left\{ x_S: |S|\le r\right\}.
\]  
Then \[\dim U_r=N(r)=\sum_{i=0}^r\binom ni.\]
\end{lemma}

\begin{lemma}\label{Layer-FunctionSpaces}
Let \(\mathbb{F}\) be a field of characteristic zero, and
\[
V_r=\operatorname{span}\left\{ g_{l,S}:0\le \ell\le n ; |S|\le r
\right\}
\]
 for
\(0\le r\le n/2\).
Then
\[
\dim V_r
=
2\sum_{i=0}^{r-1}\binom ni
+
(n-2r+1)\binom nr.
\]
\end{lemma}

\begin{proof}
For every $l$, let
\[
V_{r,l}=\operatorname{span}\left\{ g_{l,S}: |S|\le r
\right\}
\]
and
\[
\Gamma_\ell=\{A\subseteq[n]:|A|=\ell\}.
\]
We notice that if $f\in V_{r,m}$ and $A\in \Gamma_l$ for $m\neq l$, then $f(A)=0$, i.e.,  $\Gamma_l$ supports just the functions in $V_{r,l}$. Equivalently, \(V_{r,\ell}\) is the restriction of \(V_r\) to
\(\Gamma_\ell\). Thus
\begin{equation}\label{VectorSpace-Decomp}
V_r=\bigoplus_{\ell=0}^n V_{r,\ell}.
\end{equation}

Now, let \(M_{r,\ell}\) be the matrix whose rows are indexed by
\(S\subseteq[n]\) with \(|S|\le r\), whose columns are indexed by
\(A\in\Gamma_\ell\), and whose entries are
\[
M_{r,\ell}(S,A)=\1_{\{S\subseteq A\}}.
\]
Thus \(\dim V_{r,\ell}=\rank M_{r,\ell}\).

If \(\ell<r\), the rows indexed by the \(\ell\)-subsets form the
identity matrix, and hence \(M_{r,\ell}\) has full column rank which implies that
\[
\dim V_{r,\ell}=|\Gamma_{\ell}|=\binom n\ell.
\]
Suppose \(r\le \ell\). For every \(S\subseteq[n]\) with
\(|S|=s<r\), one has 
\[
\sum_{\substack{R\supseteq S\\ |R|=r}}x_R
=
\binom{\ell-s}{r-s}x_S.
\]
when two sides of the above identity are considered as functions on $\Gamma_{\ell}$.  Since the scalar in the above identity is nonzero in characteristic zero, every row of
degree less than \(r\) lies in the span of the degree-\(r\) rows.
Therefore \(M_{r,\ell}\) has the same row space as the
\(r\)-versus-\(\ell\) inclusion matrix \(I_{r,\ell}\).

By the full-rank theorem for inclusion matrices \cite{Gottlieb},
\[
\rank I_{r,\ell}
=
\min\left\{\binom nr,\binom n\ell\right\}.
\]
Consequently, in all cases,
\[
\dim V_{r,\ell}
=
\binom{n}{\min\{r,\ell,n-\ell\}}.
\]
Thus using \eqref{VectorSpace-Decomp} and \(r\le n/2\), we obtain
\begin{align*}
\dim V_r
&=
\sum_{\ell=0}^{r-1}\binom n\ell
+
\sum_{\ell=r}^{n-r}\binom nr
+
\sum_{\ell=n-r+1}^{n}\binom n{n-\ell}\\
&=
2\sum_{i=0}^{r-1}\binom ni
+
(n-2r+1)\binom nr,
\end{align*}
as claimed.
\end{proof}

\subsection{Parameterized slice-rank}

\begin{theorem}\label{thm:parameterized}
Let \(\F\subseteq2^{[n]}\) be sunflower-free.  Let \(r_A,r_B,r_C\) be integers satisfying
\[
  0\le r_A,r_B,r_C\le \frac n2
\]
and
\[
  r_A+r_B+r_C\ge n.
\]
Then
\[
  |\F|\le N(r_A)+D(r_B)+D(r_C),
\]
where
\[
  N(r)=\sum_{i=0}^r\binom ni
\]
and, for \(r\le n/2\),
\[
  D(r)=2\sum_{i=0}^{r-1}\binom ni+(n-2r+1)\binom nr.
\]
\end{theorem}

\begin{proof}
By Corollary~\ref{cor:triangular-lower}, it suffices to prove
\[
  \operatorname{sr}(T)\le N(r_A)+D(r_B)+D(r_C).
\]
Expanding
\[
  P(A,B,C)=\prod_{i=1}^n(2-a_i-b_i-c_i),
\]
every monomial in this expansion is of the form
\[
  a_I b_J c_K
  =\prod_{i\in I}a_i\prod_{j\in J}b_j\prod_{k\in K}c_k,
\]
where \(I,J,K\subseteq[n]\) are pairwise disjoint, and 
\[
  |I|+|J|+|K|\le n.
\]
Since \(r_A+r_B+r_C\ge n\), at least one of the following three inequalities must hold for every monomial 
\[
  |I|\le r_A,
  \qquad
  |J|\le r_B,
  \qquad
  |K|\le r_C.
\]

Suppose \(|I|\le r_A\) in the monomial  \(a_I b_J c_K\). Then the corresponding term of
\[
  T(A,B,C)=P(A,B,C)E(B,C)
\]
can be written as the first-coordinate slice
\[
  a_I(A)\cdot \bigl(b_J(B)c_K(C)E(B,C)\bigr)
\]
where as in the previous section 
\[
a_I(A)=\prod_{i\in I}a_i.
\]
As noted in the previous section, the functions \(a_I\) span a space of dimension at most \(N(r_A)\).

If \(|I|\ge r_A\) and  \(|J|\le r_B\), then using the fact that 
\[
  E(B,C)=\sum_{\ell=0}^n\1_{\{|B|=\ell\}}\1_{\{|C|=\ell\}}.
\]
we can slice the corresponding monomial through the \(B\)-variable using functions in
\[
  \Span\left\{
       \1_{\{|B|=\ell\}}b_J(B):0\le\ell\le n,
       \ |J|\le r_B
     \right\},
\]
whose dimension is \(D(r_B)\) by Lemma \ref{Layer-FunctionSpaces}.

Now, if \(|I|\ge r_A, |J|\ge r_B\) and \(|K| \le r_C\), similar to the above case the corresponding monomials can be sliced through the \(C\)-variable at cost at most \(D(r_C)\) slices. Therefore
\[
  \operatorname{sr}(T)\le N(r_A)+D(r_B)+D(r_C).
\]
\end{proof}

\subsection{Optimization of the slice-rank bound}
In this section, we find the minimum of the parametrized slice-rank computed in Theorem~\ref{thm:parameterized} in order to give an upper-bound for the size of the sunflower free set $\mathcal{F}$. 

Let
\[
 m:=\left\lfloor\frac n3\right\rfloor,
 \qquad
 B_n:=\binom{n}{m},
 \qquad
 t:=\left\lfloor\frac13\log_2 n\right\rfloor,
\]
and 
\[
 r_B=r_C=m-t,
 \qquad
 r_A=n-2(m-t).
\]
Then $r_A+r_B+r_C=n$, and, for all sufficiently large $n$, these integers lie in
$[0,n/2]$ satisfying the conditions of Theorem~\ref{thm:parameterized}. Hence
\[
 |\mathcal{F}|\le N(n-2(m-t))+2D(m-t).
\]
We estimate the two terms separately. Let $n=3m+\rho$ where $\rho\in\{0,1,2\}$.  For $1\le k\le m$,
\[
 \frac{\binom{n}{k-1}}{\binom{n}{k}}
 =\frac{k}{n-k+1}
 \le \frac{m}{n-m+1}<\frac12.
\]
Thus,
\[
 \binom{n}{m-t}\le 2^{-t}B_n,
 \qquad
 \sum_{i=0}^{m-t-1}\binom ni\le \binom{n}{m-t},
\]
and therefore
\[
 D(m-t)=2\sum_{i=0}^{m-t-1}\binom ni+(n-2(m-t)+1)\binom n{m-t}
 \le (n-2(m-t)+3)\binom{n}{m-t}
 \le (n+3)2^{-t}B_n.
\]
On the other hand, $r_A=n-2(m-t)=m+\rho+2t$.  Since $t=O(\log n)$, we have
$r_A\le 2n/5$ for all sufficiently large $n$.  Thus, for $1\le k\le r_A$,
\[
 \frac{\binom{n}{k-1}}{\binom{n}{k}}
 \le \frac{r_A}{n-r_A+1}\le \frac23,
\]
which implies that 
\[
 N(r_A)=\sum_{i=0}^{r_A}\binom ni\le 3\binom{n}{r_A}.
\]
Moreover, the consecutive ratios $\binom{n}{m+i}/\binom{n}{m+i-1}$ for $i\ge 1$ are at most $2$, so
\[
 \binom{n}{r_A}
 =\binom{n}{m+\rho+2t}
 \le 2^{\rho+2t}B_n
 \le 4\cdot 4^t B_n.
\]
Combining the preceding bounds, we obtain
\[
 |\mathcal{F}|
 \le \bigl(12\cdot4^t+2(n+3)2^{-t}\bigr)B_n
 =O\!\left((4^t+n2^{-t})B_n\right).
\]
Since $t=\lfloor \frac13\log_2 n\rfloor$, we have
$2^t=\Theta(n^{1/3})$, and hence
\[
 |\mathcal{F}|
 =O\!\left(n^{2/3}\binom{n}{\lfloor n/3\rfloor}\right).
\]
Finally, from Stirling's formula
\[
 \binom{n}{\lfloor n/3\rfloor}
 =\Theta\!\left(
 n^{-1/2}\left(\frac{3}{2^{2/3}}\right)^n
 \right),
\]
implying
\[
 |\mathcal{F}|
 =O\!\left(
 n^{1/6}\left(\frac{3}{2^{2/3}}\right)^n
 \right)
\]
as claimed. 
\section{Acknowledgments}
As it was mentioned in the previous sections that in~\cite{AhmadiNorouzi} we introduced the notion of higher-order $i$-triangular tensors and proved an analogue of Tao's lemma for triangular tensors of even order, and wondered whether the same result holds for triangular tensors of odd order. Asking ChatGPT for a proof of Lemma~\ref{TriangularTensor}, it came up with a proof. The proof of ChatGPT assured us of the validity of the lemma and was the starting point of this project. Our proof of Lemma~\ref{TriangularTensor} in this paper is shorter and different from that of ChatGPT.


\begin{thebibliography}{99}

\bibitem{ErdosRado}
P. Erd\H{o}s and R. Rado,
\emph{Intersection theorems for systems of sets},
J. London Math. Soc. \textbf{35} (1960), 85--90.
\url{https://doi.org/10.1112/jlms/s1-35.1.85}

\bibitem{ErdosSzemeredi1978}
P. Erd\H{o}s and E. Szemer'{e}di,
\emph{Combinatorial properties of systems of sets},
J. Combin. Theory Ser. A \textbf{24} (1978), no.~3, 308--313.
\url{https://doi.org/10.1016/0097-3165(78)90060-2}

\bibitem{NaslundSawin}
E. Naslund and W. Sawin,
\emph{Upper bounds for sunflower-free sets},
Forum Math. Sigma \textbf{5} (2017), e15, 10 pp.
\url{https://doi.org/10.1017/fms.2017.12}

\bibitem{TaoWeblog}
T. Tao,
\emph{A symmetric formulation of the
Croot--Lev--Pach--Ellenberg--Gijswijt capset bound},
What's new, 18 May 2016 (accessed 26 June 2026).
\url{https://terrytao.wordpress.com/2016/05/18/a-symmetric-formulation-of-the-croot-lev-pach-ellenberg-gijswijt-capset-bound/}

\bibitem{croot_progression-free_2016}
E. Croot, V. F. Lev, and P. P. Pach,
\emph{Progression-free sets in $\mathbb Z_4^n$ are exponentially small},
Ann. of Math. (2) \textbf{185} (2017), no.~1, 331--337.
\url{https://doi.org/10.4007/annals.2017.185.1.7}

\bibitem{ellenberg2016large}
J. S. Ellenberg and D. Gijswijt,
\emph{On large subsets of $\mathbb F_q^n$ with no three-term
arithmetic progression},
Ann. of Math. (2) \textbf{185} (2017), no.~1, 339--343.
\url{https://doi.org/10.4007/annals.2017.185.1.8}

\bibitem{Amanov}
A. Amanov and D. Yeliussizov,
\emph{Tensor slice rank and Cayley's first hyperdeterminant},
Linear Algebra Appl. \textbf{656} (2023), 224--246.
\url{https://doi.org/10.1016/j.laa.2022.09.027}

\bibitem{AhmadiNorouzi}
O. Ahmadi and H. Norouzi,
\emph{Triangular tensors and set-intersection problems},
arXiv:2508.13809v2 [math.CO], 2025.
\url{https://doi.org/10.48550/arXiv.2508.13809}

\bibitem{SawinTaoNotes}
T. Tao and W. Sawin,
\emph{Notes on the ``slice rank'' of tensors},
What's new, 24 August 2016 (accessed 26 June 2026).
\url{https://terrytao.wordpress.com/2016/08/24/notes-on-the-slice-rank-of-tensors/}

\bibitem{Gottlieb}
D. H. Gottlieb,
\emph{A certain class of incidence matrices},
Proc. Amer. Math. Soc. \textbf{17} (1966), no.~6, 1233--1237.
\url{https://doi.org/10.1090/S0002-9939-1966-0204305-9}

\end{thebibliography}
\end{document}